\documentclass{article}
\usepackage{amsmath,amssymb,amsthm}

\usepackage{tikz}
\def\picA{
\begin{tikzpicture}[scale=0.25]
    \draw[ultra thick] (-10,9)--(-5,9);
    \draw[ultra thick] (-10,8)--(-5,8);
    \draw[ultra thick] (-10,7)--(-5,7);
    \draw[ultra thick] (-10,6)--(-5,6);
    \draw[ultra thick] (-10,5)--(-5,5);
    \draw[ultra thick] (-10,4)--(-5,4);
    \draw[ultra thick] (-10,3)--(-5,3);
    \draw[ultra thick] (-10,2)--(-5,2);
    \draw[ultra thick] (-10,1)--(-5,1);
    \draw[ultra thick] (-10,0)--(-5,0);

    \node[fill=white,inner sep=1pt] at (-7.5,9) {\scriptsize $0$};
    \node[fill=white,inner sep=1pt] at (-7.5,8) {\scriptsize $1$};
    \node[fill=white,inner sep=1pt] at (-7.5,7) {\scriptsize $2$};
    \node[fill=white,inner sep=1pt] at (-7.5,6) {\scriptsize $3$};
    \node[fill=white,inner sep=1pt] at (-7.5,5) {\scriptsize $4$};
    \node[fill=white,inner sep=1pt] at (-7.5,4) {\scriptsize $5$};
    \node[fill=white,inner sep=1pt] at (-7.5,3) {\scriptsize $6$};
    \node[fill=white,inner sep=1pt] at (-7.5,2) {\scriptsize $7$};
    \node[fill=white,inner sep=1pt] at (-7.5,1) {\scriptsize $8$};
    \node[fill=white,inner sep=1pt] at (-7.5,0) {\scriptsize $9$};
    
    \draw[ultra thick] (5,6.5)--(10,6.5);
    \draw[ultra thick] (5,5.5)--(10,5.5);
    \draw[ultra thick] (5,4.5)--(10,4.5);
    \draw[ultra thick] (5,3.5)--(10,3.5);
    \draw[ultra thick] (5,2.5)--(10,2.5);
    \draw[ultra thick,dashed] (12.5,6.5)--(17.5,6.5);
    \draw[ultra thick,dashed] (12.5,5.5)--(17.5,5.5);
    \draw[ultra thick,dashed] (12.5,4.5)--(17.5,4.5);
    \draw[ultra thick,dashed] (12.5,3.5)--(17.5,3.5);
    \draw[ultra thick,dashed] (12.5,2.5)--(17.5,2.5);

    \node[fill=white,inner sep=1pt] at (7.5,6.5) {\scriptsize $0$};
    \node[fill=white,inner sep=1pt] at (7.5,5.5) {\scriptsize $1$};
    \node[fill=white,inner sep=1pt] at (7.5,4.5) {\scriptsize $2$};
    \node[fill=white,inner sep=1pt] at (7.5,3.5) {\scriptsize $3$};
    \node[fill=white,inner sep=1pt] at (7.5,2.5) {\scriptsize $4$};
    \node[fill=white,inner sep=1pt] at (15,6.5) {\scriptsize $9$};
    \node[fill=white,inner sep=1pt] at (15,5.5) {\scriptsize $8$};
    \node[fill=white,inner sep=1pt] at (15,4.5) {\scriptsize $7$};
    \node[fill=white,inner sep=1pt] at (15,3.5) {\scriptsize $6$};
    \node[fill=white,inner sep=1pt] at (15,2.5) {\scriptsize $5$};
    
    \draw[ultra thick] (27.5,9)--(32.5,9);
    \draw[ultra thick] (27.5,7)--(32.5,7);
    \draw[ultra thick] (27.5,5)--(32.5,5);
    \draw[ultra thick] (27.5,3)--(32.5,3);
    \draw[ultra thick] (27.5,1)--(32.5,1);
    \draw[ultra thick,dashed] (28.75,8)--(33.75,8);
    \draw[ultra thick,dashed] (28.75,6)--(33.75,6);
    \draw[ultra thick,dashed] (28.75,4)--(33.75,4);
    \draw[ultra thick,dashed] (28.75,2)--(33.75,2);
    \draw[ultra thick,dashed] (28.75,0)--(33.75,0);

    \node[fill=white,inner sep=1pt] at (30,9) {\scriptsize $0$};
    \node[fill=white,inner sep=1pt] at (30,7) {\scriptsize $1$};
    \node[fill=white,inner sep=1pt] at (30,5) {\scriptsize $2$};
    \node[fill=white,inner sep=1pt] at (30,3) {\scriptsize $3$};
    \node[fill=white,inner sep=1pt] at (30,1) {\scriptsize $4$};
    \node[fill=white,inner sep=1pt] at (31.25,8) {\scriptsize $9$};
    \node[fill=white,inner sep=1pt] at (31.25,6) {\scriptsize $8$};
    \node[fill=white,inner sep=1pt] at (31.25,4) {\scriptsize $7$};
    \node[fill=white,inner sep=1pt] at (31.25,2) {\scriptsize $6$};
    \node[fill=white,inner sep=1pt] at (31.25,0) {\scriptsize $5$};

    \draw[->,  thick] (-2.5,4.5)--(2.5,4.5);
    \draw[->,  thick] (20,4.5)--(25,4.5);

    \node at (0,2.5) {\small cut \& flip};
    \node at (22.5,2.5) {\small interlace};
\end{tikzpicture}}

\def\picB{
\begin{tikzpicture}[scale=0.25]
    \draw[ultra thick] (-10,9)--(-5,9);
    \draw[ultra thick] (-10,8)--(-5,8);
    \draw[ultra thick] (-10,7)--(-5,7);
    \draw[ultra thick] (-10,6)--(-5,6);
    \draw[ultra thick] (-10,5)--(-5,5);
    \draw[ultra thick] (-10,4)--(-5,4);
    \draw[ultra thick] (-10,3)--(-5,3);
    \draw[ultra thick] (-10,2)--(-5,2);
    \draw[ultra thick] (-10,1)--(-5,1);
    \draw[ultra thick] (-10,0)--(-5,0);

    \node[fill=white,inner sep=1pt] at (-7.5,9) {\scriptsize $0$};
    \node[fill=white,inner sep=1pt] at (-7.5,8) {\scriptsize $1$};
    \node[fill=white,inner sep=1pt] at (-7.5,7) {\scriptsize $2$};
    \node[fill=white,inner sep=1pt] at (-7.5,6) {\scriptsize $3$};
    \node[fill=white,inner sep=1pt] at (-7.5,5) {\scriptsize $4$};
    \node[fill=white,inner sep=1pt] at (-7.5,4) {\scriptsize $5$};
    \node[fill=white,inner sep=1pt] at (-7.5,3) {\scriptsize $6$};
    \node[fill=white,inner sep=1pt] at (-7.5,2) {\scriptsize $7$};
    \node[fill=white,inner sep=1pt] at (-7.5,1) {\scriptsize $8$};
    \node[fill=white,inner sep=1pt] at (-7.5,0) {\scriptsize $9$};
    
    \draw[ultra thick,dashed] (5,6.5)--(10,6.5);
    \draw[ultra thick,dashed] (5,5.5)--(10,5.5);
    \draw[ultra thick,dashed] (5,4.5)--(10,4.5);
    \draw[ultra thick,dashed] (5,3.5)--(10,3.5);
    \draw[ultra thick,dashed] (5,2.5)--(10,2.5);
    \draw[ultra thick,dashed] (12.5,6.5)--(17.5,6.5);
    \draw[ultra thick,dashed] (12.5,5.5)--(17.5,5.5);
    \draw[ultra thick,dashed] (12.5,4.5)--(17.5,4.5);
    \draw[ultra thick,dashed] (12.5,3.5)--(17.5,3.5);
    \draw[ultra thick,dashed] (12.5,2.5)--(17.5,2.5);

    \node[fill=white,inner sep=1pt] at (7.5,6.5) {\scriptsize $8$};
    \node[fill=white,inner sep=1pt] at (7.5,5.5) {\scriptsize $6$};
    \node[fill=white,inner sep=1pt] at (7.5,4.5) {\scriptsize $4$};
    \node[fill=white,inner sep=1pt] at (7.5,3.5) {\scriptsize $2$};
    \node[fill=white,inner sep=1pt] at (7.5,2.5) {\scriptsize $0$};
    \node[fill=white,inner sep=1pt] at (15,6.5) {\scriptsize $9$};
    \node[fill=white,inner sep=1pt] at (15,5.5) {\scriptsize $7$};
    \node[fill=white,inner sep=1pt] at (15,4.5) {\scriptsize $5$};
    \node[fill=white,inner sep=1pt] at (15,3.5) {\scriptsize $3$};
    \node[fill=white,inner sep=1pt] at (15,2.5) {\scriptsize $1$};
    
    \draw[ultra thick] (27.5,9)--(32.5,9);
    \draw[ultra thick] (27.5,8)--(32.5,8);
    \draw[ultra thick] (27.5,7)--(32.5,7);
    \draw[ultra thick] (27.5,6)--(32.5,6);
    \draw[ultra thick] (27.5,5)--(32.5,5);
    \draw[ultra thick,dashed] (27.5,4)--(32.5,4);
    \draw[ultra thick,dashed] (27.5,3)--(32.5,3);
    \draw[ultra thick,dashed] (27.5,2)--(32.5,2);
    \draw[ultra thick,dashed] (27.5,1)--(32.5,1);
    \draw[ultra thick,dashed] (27.5,0)--(32.5,0);

    \node[fill=white,inner sep=1pt] at (30,9) {\scriptsize $0$};
    \node[fill=white,inner sep=1pt] at (30,8) {\scriptsize $2$};
    \node[fill=white,inner sep=1pt] at (30,7) {\scriptsize $4$};
    \node[fill=white,inner sep=1pt] at (30,6) {\scriptsize $6$};
    \node[fill=white,inner sep=1pt] at (30,5) {\scriptsize $8$};
    \node[fill=white,inner sep=1pt] at (30,4) {\scriptsize $9$};
    \node[fill=white,inner sep=1pt] at (30,3) {\scriptsize $7$};
    \node[fill=white,inner sep=1pt] at (30,2) {\scriptsize $5$};
    \node[fill=white,inner sep=1pt] at (30,1) {\scriptsize $3$};
    \node[fill=white,inner sep=1pt] at (30,0) {\scriptsize $1$};

    \draw[->,  thick] (-2.5,4.5)--(2.5,4.5);
    \draw[->,  thick] (20,4.5)--(25,4.5);

    \node at (0,2.5) {\small into piles};
    \node at (22.5,2.5) {\small flip over};
\end{tikzpicture}}

\def\picC{
\begin{tikzpicture}[scale=0.25]
    \draw[ultra thick] (-10,9)--(-5,9);
    \draw[ultra thick] (-10,8)--(-5,8);
    \draw[ultra thick] (-10,7)--(-5,7);
    \draw[ultra thick] (-10,6)--(-5,6);
    \draw[ultra thick] (-10,5)--(-5,5);
    \draw[ultra thick] (-10,4)--(-5,4);
    \draw[ultra thick] (-10,3)--(-5,3);
    \draw[ultra thick] (-10,2)--(-5,2);
    \draw[ultra thick] (-10,1)--(-5,1);
    \draw[ultra thick] (-10,0)--(-5,0);

    \node[fill=white,inner sep=1pt] at (-7.5,9) {\scriptsize $0$};
    \node[fill=white,inner sep=1pt] at (-7.5,8) {\scriptsize $1$};
    \node[fill=white,inner sep=1pt] at (-7.5,7) {\scriptsize $2$};
    \node[fill=white,inner sep=1pt] at (-7.5,6) {\scriptsize $3$};
    \node[fill=white,inner sep=1pt] at (-7.5,5) {\scriptsize $4$};
    \node[fill=white,inner sep=1pt] at (-7.5,4) {\scriptsize $5$};
    \node[fill=white,inner sep=1pt] at (-7.5,3) {\scriptsize $6$};
    \node[fill=white,inner sep=1pt] at (-7.5,2) {\scriptsize $7$};
    \node[fill=white,inner sep=1pt] at (-7.5,1) {\scriptsize $8$};
    \node[fill=white,inner sep=1pt] at (-7.5,0) {\scriptsize $9$};
    
    \draw[ultra thick] (5,6.5)--(10,6.5);
    \draw[ultra thick] (5,5.5)--(10,5.5);
    \draw[ultra thick] (5,4.5)--(10,4.5);
    \draw[ultra thick] (5,3.5)--(10,3.5);
    \draw[ultra thick] (5,2.5)--(10,2.5);
    \draw[ultra thick,dashed] (12.5,6.5)--(17.5,6.5);
    \draw[ultra thick,dashed] (12.5,5.5)--(17.5,5.5);
    \draw[ultra thick,dashed] (12.5,4.5)--(17.5,4.5);
    \draw[ultra thick,dashed] (12.5,3.5)--(17.5,3.5);
    \draw[ultra thick,dashed] (12.5,2.5)--(17.5,2.5);

    \node[fill=white,inner sep=1pt] at (7.5,6.5) {\scriptsize $0$};
    \node[fill=white,inner sep=1pt] at (7.5,5.5) {\scriptsize $1$};
    \node[fill=white,inner sep=1pt] at (7.5,4.5) {\scriptsize $2$};
    \node[fill=white,inner sep=1pt] at (7.5,3.5) {\scriptsize $3$};
    \node[fill=white,inner sep=1pt] at (7.5,2.5) {\scriptsize $4$};
    \node[fill=white,inner sep=1pt] at (15,6.5) {\scriptsize $9$};
    \node[fill=white,inner sep=1pt] at (15,5.5) {\scriptsize $8$};
    \node[fill=white,inner sep=1pt] at (15,4.5) {\scriptsize $7$};
    \node[fill=white,inner sep=1pt] at (15,3.5) {\scriptsize $6$};
    \node[fill=white,inner sep=1pt] at (15,2.5) {\scriptsize $5$};
    
    \draw[ultra thick,dashed] (27.5,9)--(32.5,9);
    \draw[ultra thick,dashed] (27.5,7)--(32.5,7);
    \draw[ultra thick,dashed] (27.5,5)--(32.5,5);
    \draw[ultra thick,dashed] (27.5,3)--(32.5,3);
    \draw[ultra thick,dashed] (27.5,1)--(32.5,1);
    \draw[ultra thick] (28.75,8)--(33.75,8);
    \draw[ultra thick] (28.75,6)--(33.75,6);
    \draw[ultra thick] (28.75,4)--(33.75,4);
    \draw[ultra thick] (28.75,2)--(33.75,2);
    \draw[ultra thick] (28.75,0)--(33.75,0);

    \node[fill=white,inner sep=1pt] at (31.25,8) {\scriptsize $0$};
    \node[fill=white,inner sep=1pt] at (31.25,6) {\scriptsize $1$};
    \node[fill=white,inner sep=1pt] at (31.25,4) {\scriptsize $2$};
    \node[fill=white,inner sep=1pt] at (31.25,2) {\scriptsize $3$};
    \node[fill=white,inner sep=1pt] at (31.25,0) {\scriptsize $4$};
    \node[fill=white,inner sep=1pt] at (30,9) {\scriptsize $9$};
    \node[fill=white,inner sep=1pt] at (30,7) {\scriptsize $8$};
    \node[fill=white,inner sep=1pt] at (30,5) {\scriptsize $7$};
    \node[fill=white,inner sep=1pt] at (30,3) {\scriptsize $6$};
    \node[fill=white,inner sep=1pt] at (30,1) {\scriptsize $5$};

    \draw[->,  thick] (-2.5,4.5)--(2.5,4.5);
    \draw[->,  thick] (20,4.5)--(25,4.5);

    \node at (0,2.5) {\small cut \& flip};
    \node at (22.5,2.5) {\small interlace};
\end{tikzpicture}}

\def\picD{
\begin{tikzpicture}[scale=0.25]
    \draw[ultra thick] (-10,9)--(-5,9);
    \draw[ultra thick] (-10,8)--(-5,8);
    \draw[ultra thick] (-10,7)--(-5,7);
    \draw[ultra thick] (-10,6)--(-5,6);
    \draw[ultra thick] (-10,5)--(-5,5);
    \draw[ultra thick] (-10,4)--(-5,4);
    \draw[ultra thick] (-10,3)--(-5,3);
    \draw[ultra thick] (-10,2)--(-5,2);
    \draw[ultra thick] (-10,1)--(-5,1);
    \draw[ultra thick] (-10,0)--(-5,0);

    \node[fill=white,inner sep=1pt] at (-7.5,9) {\scriptsize $0$};
    \node[fill=white,inner sep=1pt] at (-7.5,8) {\scriptsize $1$};
    \node[fill=white,inner sep=1pt] at (-7.5,7) {\scriptsize $2$};
    \node[fill=white,inner sep=1pt] at (-7.5,6) {\scriptsize $3$};
    \node[fill=white,inner sep=1pt] at (-7.5,5) {\scriptsize $4$};
    \node[fill=white,inner sep=1pt] at (-7.5,4) {\scriptsize $5$};
    \node[fill=white,inner sep=1pt] at (-7.5,3) {\scriptsize $6$};
    \node[fill=white,inner sep=1pt] at (-7.5,2) {\scriptsize $7$};
    \node[fill=white,inner sep=1pt] at (-7.5,1) {\scriptsize $8$};
    \node[fill=white,inner sep=1pt] at (-7.5,0) {\scriptsize $9$};
    
    \draw[ultra thick,dashed] (5,6.5)--(10,6.5);
    \draw[ultra thick,dashed] (5,5.5)--(10,5.5);
    \draw[ultra thick,dashed] (5,4.5)--(10,4.5);
    \draw[ultra thick,dashed] (5,3.5)--(10,3.5);
    \draw[ultra thick,dashed] (5,2.5)--(10,2.5);
    \draw[ultra thick,dashed] (12.5,6.5)--(17.5,6.5);
    \draw[ultra thick,dashed] (12.5,5.5)--(17.5,5.5);
    \draw[ultra thick,dashed] (12.5,4.5)--(17.5,4.5);
    \draw[ultra thick,dashed] (12.5,3.5)--(17.5,3.5);
    \draw[ultra thick,dashed] (12.5,2.5)--(17.5,2.5);

    \node[fill=white,inner sep=1pt] at (7.5,6.5) {\scriptsize $8$};
    \node[fill=white,inner sep=1pt] at (7.5,5.5) {\scriptsize $6$};
    \node[fill=white,inner sep=1pt] at (7.5,4.5) {\scriptsize $4$};
    \node[fill=white,inner sep=1pt] at (7.5,3.5) {\scriptsize $2$};
    \node[fill=white,inner sep=1pt] at (7.5,2.5) {\scriptsize $0$};
    \node[fill=white,inner sep=1pt] at (15,6.5) {\scriptsize $9$};
    \node[fill=white,inner sep=1pt] at (15,5.5) {\scriptsize $7$};
    \node[fill=white,inner sep=1pt] at (15,4.5) {\scriptsize $5$};
    \node[fill=white,inner sep=1pt] at (15,3.5) {\scriptsize $3$};
    \node[fill=white,inner sep=1pt] at (15,2.5) {\scriptsize $1$};
    
    \draw[ultra thick] (27.5,9)--(32.5,9);
    \draw[ultra thick] (27.5,8)--(32.5,8);
    \draw[ultra thick] (27.5,7)--(32.5,7);
    \draw[ultra thick] (27.5,6)--(32.5,6);
    \draw[ultra thick] (27.5,5)--(32.5,5);
    \draw[ultra thick,dashed] (27.5,4)--(32.5,4);
    \draw[ultra thick,dashed] (27.5,3)--(32.5,3);
    \draw[ultra thick,dashed] (27.5,2)--(32.5,2);
    \draw[ultra thick,dashed] (27.5,1)--(32.5,1);
    \draw[ultra thick,dashed] (27.5,0)--(32.5,0);

    \node[fill=white,inner sep=1pt] at (30,9) {\scriptsize $1$};
    \node[fill=white,inner sep=1pt] at (30,8) {\scriptsize $3$};
    \node[fill=white,inner sep=1pt] at (30,7) {\scriptsize $5$};
    \node[fill=white,inner sep=1pt] at (30,6) {\scriptsize $7$};
    \node[fill=white,inner sep=1pt] at (30,5) {\scriptsize $9$};
    \node[fill=white,inner sep=1pt] at (30,4) {\scriptsize $8$};
    \node[fill=white,inner sep=1pt] at (30,3) {\scriptsize $6$};
    \node[fill=white,inner sep=1pt] at (30,2) {\scriptsize $4$};
    \node[fill=white,inner sep=1pt] at (30,1) {\scriptsize $2$};
    \node[fill=white,inner sep=1pt] at (30,0) {\scriptsize $0$};

    \draw[->,  thick] (-2.5,4.5)--(2.5,4.5);
    \draw[->,  thick] (20,4.5)--(25,4.5);

    \node at (0,2.5) {\small into piles};
    \node at (22.5,2.5) {\small flip over};
\end{tikzpicture}}

\def\picE{\begin{tikzpicture}[scale=0.25]
\foreach \x in {0,...,9}
   {\draw[ultra thick] (-2.5,10-\x-1)--(2.5,10-\x-1);  \node[fill=white, inner sep=1pt] at (0,10-\x-1) {\scriptsize \x};}
   
\foreach \x in {1,2,3,4,0}
   {\draw[ultra thick] (-18.75,11-2*\x-2)--(-13.75,11-2*\x-2); \node[fill=white, inner sep=1pt] at (-16.25,11-2*\x-2) {\scriptsize \x};
    \draw[ultra thick] (12.5,10-2*\x-2)--(17.5,10-2*\x-2); \node[fill=white, inner sep=1pt] at (15,10-2*\x-2) {\scriptsize \x}; }

\foreach \x in {6,7,8,9,5}
   {\draw[ultra thick] (-17.5,20-2*\x-2)--(-12.5,20-2*\x-2); \node[fill=white, inner sep=1pt] at (-15,20-2*\x-2) {\scriptsize \x};
    \draw[ultra thick] (13.75,21-2*\x-2)--(18.75,21-2*\x-2); \node[fill=white, inner sep=1pt] at (16.25,21-2*\x-2) {\scriptsize \x}; }

\draw[->,  thick] (-5,4.5)--(-10,4.5);
\draw[->,  thick] (5,4.5)--(10,4.5);
\node at (7.5,3) {\small in};
\node at (-7.5,3) {\small out};
\end{tikzpicture}}

\def\picF{\begin{tikzpicture}[scale=0.25]
\foreach \x in {0,...,9}
   {\draw[ultra thick] (-2.5,10-\x-1)--(2.5,10-\x-1);  \node[fill=white, inner sep=1pt] at (0,10-\x-1) {\scriptsize \x};}

\foreach \x in {1,2,3,4,0}
   {\draw[ultra thick] (-18.75,11-2*\x-2)--(-13.75,11-2*\x-2); \node[fill=white, inner sep=1pt] at (-16.25,11-2*\x-2) {\scriptsize \x}; 
    \draw[ultra thick] (12.5,10-2*\x-2)--(17.5,10-2*\x-2); \node[fill=white, inner sep=1pt] at (15,10-2*\x-2) {\scriptsize \x}; }

\foreach \x in {6,7,8,9,5}
   {\draw[ultra thick] (-17.5,2*\x-12+2)--(-12.5,2*\x-12+2); \node[fill=white, inner sep=1pt] at (-15,2*\x-12+2) {\scriptsize \x};
    \draw[ultra thick] (13.75,2*\x-11+2)--(18.75,2*\x-11+2); \node[fill=white, inner sep=1pt] at (16.25,2*\x-11+2) {\scriptsize \x}; }

\draw[->,  thick] (-5,4.5)--(-10,4.5);
\draw[->,  thick] (5,4.5)--(10,4.5);
\node at (7.5,3) {\small in};
\node at (-7.5,3) {\small out};
\end{tikzpicture}}

\def\picG{\begin{tikzpicture}[scale=0.25]
\foreach \x in {0,...,9}
   {\draw[ultra thick] (-2.5,10-\x-1)--(2.5,10-\x-1);  \node[fill=white, inner sep=1pt] at (0,10-\x-1) {\scriptsize \x};}
\foreach \x in {1,2,3,4,0}
   {
    \draw[ultra thick] (12.5,10-2*\x-2)--(17.5,10-2*\x-2); }

\foreach \x in {6,7,8,9,5}
   {    \draw[ultra thick] (13.75,21-2*\x-2)--(18.75,21-2*\x-2); }
   
\draw[->,  thick] (5,4.5)--(10,4.5);
\node at (7.5,3) {\small milk};

\node[fill=white, inner sep=1pt] at (16.25,9) {\scriptsize $4$};
\node[fill=white, inner sep=1pt] at (15  ,8) {\scriptsize $5$};
\node[fill=white, inner sep=1pt] at (16.25,7) {\scriptsize $3$};
\node[fill=white, inner sep=1pt] at (15  ,6) {\scriptsize $6$};
\node[fill=white, inner sep=1pt] at (16.25,5) {\scriptsize $2$};
\node[fill=white, inner sep=1pt] at (15  ,4) {\scriptsize $7$};
\node[fill=white, inner sep=1pt] at (16.25,3) {\scriptsize $1$};
\node[fill=white, inner sep=1pt] at (15  ,2) {\scriptsize $8$};
\node[fill=white, inner sep=1pt] at (16.25,1) {\scriptsize $0$};
\node[fill=white, inner sep=1pt] at (15  ,0) {\scriptsize $9$};

\end{tikzpicture}}

\def\picH{\begin{tikzpicture}[scale=0.25]
\foreach \x in {0,...,9}
   {\draw[ultra thick] (-2.5,10-\x-1)--(2.5,10-\x-1);  \node[fill=white, inner sep=1pt] at (0,10-\x-1) {\scriptsize \x};}
\foreach \x in {1,2,3,4,0}
   {\draw[ultra thick] (-18.75,11-2*\x-2)--(-13.75,11-2*\x-2); 
    \draw[ultra thick] (12.5,10-2*\x-2)--(17.5,10-2*\x-2);  }
\foreach \x in {6,7,8,9,5}
   {\draw[ultra thick] (-17.5,2*\x-12+2)--(-12.5,2*\x-12+2); ;
    \draw[ultra thick] (13.75,2*\x-11+2)--(18.75,2*\x-11+2);  }

\draw[->,  thick] (-5,4.5)--(-10,4.5);
\draw[->,  thick] (5,4.5)--(10,4.5);

\node at (7.5,3) {\small $2^{\text{nd}}$ under};
\node at (-7.5,3) {\small $2^{\text{nd}}$ over};

\node[fill=white, inner sep=1pt] at (16.25,9) {\scriptsize $8$};
\node[fill=white, inner sep=1pt] at (15  ,8) {\scriptsize $6$};
\node[fill=white, inner sep=1pt] at (16.25,7) {\scriptsize $4$};
\node[fill=white, inner sep=1pt] at (15  ,6) {\scriptsize $2$};
\node[fill=white, inner sep=1pt] at (16.25,5) {\scriptsize $0$};
\node[fill=white, inner sep=1pt] at (15  ,4) {\scriptsize $1$};
\node[fill=white, inner sep=1pt] at (16.25,3) {\scriptsize $3$};
\node[fill=white, inner sep=1pt] at (15  ,2) {\scriptsize $5$};
\node[fill=white, inner sep=1pt] at (16.25,1) {\scriptsize $7$};
\node[fill=white, inner sep=1pt] at (15  ,0) {\scriptsize $9$};

\node[fill=white, inner sep=1pt] at (-16.25  ,9) {\scriptsize $9$};
\node[fill=white, inner sep=1pt] at (-15,8) {\scriptsize $7$};
\node[fill=white, inner sep=1pt] at (-16.25  ,7) {\scriptsize $5$};
\node[fill=white, inner sep=1pt] at (-15,6) {\scriptsize $3$};
\node[fill=white, inner sep=1pt] at (-16.25  ,5) {\scriptsize $1$};
\node[fill=white, inner sep=1pt] at (-15,4) {\scriptsize $0$};
\node[fill=white, inner sep=1pt] at (-16.25  ,3) {\scriptsize $2$};
\node[fill=white, inner sep=1pt] at (-15,2) {\scriptsize $4$};
\node[fill=white, inner sep=1pt] at (-16.25  ,1) {\scriptsize $6$};
\node[fill=white, inner sep=1pt] at (-15,0) {\scriptsize $8$};

\end{tikzpicture}}

\def\picAA{\begin{tikzpicture}[scale=1.25]
\node[rounded corners, rectangle,draw=black,fill=white, thick] (1bit) at (2,0) {$2^0$ bit};
\node[rounded corners, rectangle,draw=black,fill=white, thick] (2bit) at (4,0) {$2^1$ bit};
\node[rounded corners, rectangle,draw=black,fill=white, thick] (3bit) at (6,0) {$2^2$ bit};
\node[] (midbit) at (8,0) {$\cdots$};
\node[rounded corners, rectangle,draw=black,fill=white, thick] (lastbit) at (10,0) {$2^{k-1}$ bit};
\node[rounded corners, rectangle,draw=black,fill=white, thick] (comp) at (6,-1) {complement};
\draw[->, ultra thick] (1bit.east)--(2bit.west);
\draw[->, ultra thick] (2bit.east)--(3bit.west);
\draw[->, ultra thick] (3bit.east)--(midbit.west);
\draw[->, ultra thick] (midbit.east)--(lastbit.west);
\draw[->, ultra thick, rounded corners] (lastbit.south)--(10,-1)--(comp.east);
\draw[->, ultra thick, rounded corners] (comp.west)--(2,-1)--(1bit.south);
\end{tikzpicture}}

\def\picBB{\begin{tikzpicture}[scale=0.45]
\draw[ultra thick] (-10,7.5)--(-6,11.5);
\draw[ultra thick] (-10,7.5)--(-6,3.5);
\draw[ultra thick] (-6,11.5)--(-2,13.5);
\draw[ultra thick] (-6,11.5)--(-2,9.5);
\draw[ultra thick] (-6,3.5)--(-2,5.5);
\draw[ultra thick] (-6,3.5)--(-2,1.5);
\draw[ultra thick] (-2,13.5)--(2,14.5);
\draw[ultra thick] (-2,13.5)--(2,12.5);
\draw[ultra thick] (-2,9.5)--(2,10.5);
\draw[ultra thick] (-2,9.5)--(2,8.5);
\draw[ultra thick] (-2,5.5)--(2,6.5);
\draw[ultra thick] (-2,5.5)--(2,4.5);
\draw[ultra thick] (-2,1.5)--(2,2.5);
\draw[ultra thick] (-2,1.5)--(2,0.5);
\draw[ultra thick] (2,14.5)--(6,15);
\draw[ultra thick] (2,14.5)--(6,14);
\draw[ultra thick] (2,12.5)--(6,13);
\draw[ultra thick] (2,12.5)--(6,12);
\draw[ultra thick] (2,10.5)--(6,11);
\draw[ultra thick] (2,10.5)--(6,10);
\draw[ultra thick] (2,8.5)--(6,9);
\draw[ultra thick] (2,8.5)--(6,8);
\draw[ultra thick] (2,6.5)--(6,7);
\draw[ultra thick] (2,6.5)--(6,6);
\draw[ultra thick] (2,4.5)--(6,5);
\draw[ultra thick] (2,4.5)--(6,4);
\draw[ultra thick] (2,2.5)--(6,3);
\draw[ultra thick] (2,2.5)--(6,2);
\draw[ultra thick] (2,0.5)--(6,1);
\draw[ultra thick] (2,0.5)--(6,0);

\node[fill=white, inner sep=2pt] at (6,15) {$0101$};
\node[fill=white, inner sep=2pt] at (6,14) {$0100$};
\node[fill=white, inner sep=2pt] at (6,13) {$0111$};
\node[fill=white, inner sep=2pt] at (6,12) {$0110$};
\node[fill=white, inner sep=2pt] at (6,11) {$0001$};
\node[fill=white, inner sep=2pt] at (6,10) {$0000$};
\node[fill=white, inner sep=2pt] at (6,9) {$0011$};
\node[fill=white, inner sep=2pt] at (6,8) {$0010$};
\node[fill=white, inner sep=2pt] at (6,7) {$1101$};
\node[fill=white, inner sep=2pt] at (6,6) {$1100$};
\node[fill=white, inner sep=2pt] at (6,5) {$1111$};
\node[fill=white, inner sep=2pt] at (6,4) {$1110$};
\node[fill=white, inner sep=2pt] at (6,3) {$1001$};
\node[fill=white, inner sep=2pt] at (6,2) {$1000$};
\node[fill=white, inner sep=2pt] at (6,1) {$1011$};
\node[fill=white, inner sep=2pt] at (6,0) {$1010$};

\node[fill=white, inner sep=2pt] at (2,14.5) {$0010$};
\node[fill=white, inner sep=2pt] at (2,12.5) {$0011$};
\node[fill=white, inner sep=2pt] at (2,10.5) {$0000$};
\node[fill=white, inner sep=2pt] at (2,8.5) {$0001$};
\node[fill=white, inner sep=2pt] at (2,6.5) {$0110$};
\node[fill=white, inner sep=2pt] at (2,4.5) {$0111$};
\node[fill=white, inner sep=2pt] at (2,2.5) {$0100$};
\node[fill=white, inner sep=2pt] at (2,0.5) {$0101$};

\node[fill=white, inner sep=2pt] at (-2,13.5) {$1110$};
\node[fill=white, inner sep=2pt] at (-2,9.5) {$1111$};
\node[fill=white, inner sep=2pt] at (-2,5.5) {$1100$};
\node[fill=white, inner sep=2pt] at (-2,1.5) {$1101$};

\node[fill=white, inner sep=2pt] at (-6,11.5) {$1000$};
\node[fill=white, inner sep=2pt] at (-6,3.5) {$1001$};

\node[fill=white, inner sep=2pt] at (-10,7.5) {$1011$};
\end{tikzpicture}}

\def\picAAA{
\begin{tikzpicture}[scale=0.25]
    \draw[ultra thick] (-10,9)--(-5,9);
    \draw[ultra thick] (-10,8)--(-5,8);
    \draw[ultra thick] (-10,7)--(-5,7);
    \draw[ultra thick] (-10,6)--(-5,6);
    \draw[ultra thick] (-10,5)--(-5,5);
    \draw[ultra thick] (-10,4)--(-5,4);
    \draw[ultra thick] (-10,3)--(-5,3);
    \draw[ultra thick] (-10,2)--(-5,2);
    \draw[ultra thick] (-10,1)--(-5,1);
    \draw[ultra thick] (-10,0)--(-5,0);

    \node[fill=white,inner sep=1pt] at (-7.5,9) {\scriptsize $0$};
    \node[fill=white,inner sep=1pt] at (-7.5,8) {\scriptsize $1$};
    \node[fill=white,inner sep=1pt] at (-7.5,7) {\scriptsize $2$};
    \node[fill=white,inner sep=1pt] at (-7.5,6) {\scriptsize $3$};
    \node[fill=white,inner sep=1pt] at (-7.5,5) {\scriptsize $4$};
    \node[fill=white,inner sep=1pt] at (-7.5,4) {\scriptsize $5$};
    \node[fill=white,inner sep=1pt] at (-7.5,3) {\scriptsize $6$};
    \node[fill=white,inner sep=1pt] at (-7.5,2) {\scriptsize $7$};
    \node[fill=white,inner sep=1pt] at (-7.5,1) {\scriptsize $8$};
    \node[fill=white,inner sep=1pt] at (-7.5,0) {\scriptsize $9$};
    
    \draw[ultra thick,dashed] (5,6.5)--(10,6.5);
    \draw[ultra thick,dashed] (5,5.5)--(10,5.5);
    \draw[ultra thick,dashed] (5,4.5)--(10,4.5);
    \draw[ultra thick,dashed] (5,3.5)--(10,3.5);
    \draw[ultra thick,dashed] (5,2.5)--(10,2.5);
    \draw[ultra thick] (12.5,6.5)--(17.5,6.5);
    \draw[ultra thick] (12.5,5.5)--(17.5,5.5);
    \draw[ultra thick] (12.5,4.5)--(17.5,4.5);
    \draw[ultra thick] (12.5,3.5)--(17.5,3.5);
    \draw[ultra thick] (12.5,2.5)--(17.5,2.5);

    \node[fill=white,inner sep=1pt] at (7.5,6.5) {\scriptsize $8$};
    \node[fill=white,inner sep=1pt] at (7.5,5.5) {\scriptsize $6$};
    \node[fill=white,inner sep=1pt] at (7.5,4.5) {\scriptsize $4$};
    \node[fill=white,inner sep=1pt] at (7.5,3.5) {\scriptsize $2$};
    \node[fill=white,inner sep=1pt] at (7.5,2.5) {\scriptsize $0$};
    \node[fill=white,inner sep=1pt] at (15,6.5) {\scriptsize $9$};
    \node[fill=white,inner sep=1pt] at (15,5.5) {\scriptsize $7$};
    \node[fill=white,inner sep=1pt] at (15,4.5) {\scriptsize $5$};
    \node[fill=white,inner sep=1pt] at (15,3.5) {\scriptsize $3$};
    \node[fill=white,inner sep=1pt] at (15,2.5) {\scriptsize $1$};
    
    \draw[ultra thick] (27.5,9)--(32.5,9);
    \draw[ultra thick] (27.5,8)--(32.5,8);
    \draw[ultra thick] (27.5,7)--(32.5,7);
    \draw[ultra thick] (27.5,6)--(32.5,6);
    \draw[ultra thick] (27.5,5)--(32.5,5);
    \draw[ultra thick] (27.5,4)--(32.5,4);
    \draw[ultra thick] (27.5,3)--(32.5,3);
    \draw[ultra thick] (27.5,2)--(32.5,2);
    \draw[ultra thick] (27.5,1)--(32.5,1);
    \draw[ultra thick] (27.5,0)--(32.5,0);

    \node[fill=white,inner sep=1pt] at (30,9) {\scriptsize $0$};
    \node[fill=white,inner sep=1pt] at (30,8) {\scriptsize $2$};
    \node[fill=white,inner sep=1pt] at (30,7) {\scriptsize $4$};
    \node[fill=white,inner sep=1pt] at (30,6) {\scriptsize $6$};
    \node[fill=white,inner sep=1pt] at (30,5) {\scriptsize $8$};
    \node[fill=white,inner sep=1pt] at (30,4) {\scriptsize $9$};
    \node[fill=white,inner sep=1pt] at (30,3) {\scriptsize $7$};
    \node[fill=white,inner sep=1pt] at (30,2) {\scriptsize $5$};
    \node[fill=white,inner sep=1pt] at (30,1) {\scriptsize $3$};
    \node[fill=white,inner sep=1pt] at (30,0) {\scriptsize $1$};

    \draw[->,  thick] (-2.5,4.5)--(2.5,4.5);
    \draw[->,  thick] (20,4.5)--(25,4.5);

    \node at (0,2.5) {\small into piles};
    \node at (22.5,2.5) {\small flip over};
\end{tikzpicture}}

\def\picBBB{
\begin{tikzpicture}[scale=0.25]
    \draw[ultra thick] (-10,9)--(-5,9);
    \draw[ultra thick] (-10,8)--(-5,8);
    \draw[ultra thick] (-10,7)--(-5,7);
    \draw[ultra thick] (-10,6)--(-5,6);
    \draw[ultra thick] (-10,5)--(-5,5);
    \draw[ultra thick] (-10,4)--(-5,4);
    \draw[ultra thick] (-10,3)--(-5,3);
    \draw[ultra thick] (-10,2)--(-5,2);
    \draw[ultra thick] (-10,1)--(-5,1);
    \draw[ultra thick] (-10,0)--(-5,0);

    \node[fill=white,inner sep=1pt] at (-7.5,9) {\scriptsize $0$};
    \node[fill=white,inner sep=1pt] at (-7.5,8) {\scriptsize $1$};
    \node[fill=white,inner sep=1pt] at (-7.5,7) {\scriptsize $2$};
    \node[fill=white,inner sep=1pt] at (-7.5,6) {\scriptsize $3$};
    \node[fill=white,inner sep=1pt] at (-7.5,5) {\scriptsize $4$};
    \node[fill=white,inner sep=1pt] at (-7.5,4) {\scriptsize $5$};
    \node[fill=white,inner sep=1pt] at (-7.5,3) {\scriptsize $6$};
    \node[fill=white,inner sep=1pt] at (-7.5,2) {\scriptsize $7$};
    \node[fill=white,inner sep=1pt] at (-7.5,1) {\scriptsize $8$};
    \node[fill=white,inner sep=1pt] at (-7.5,0) {\scriptsize $9$};
    
    \draw[ultra thick] (5,6.5)--(10,6.5);
    \draw[ultra thick] (5,5.5)--(10,5.5);
    \draw[ultra thick] (5,4.5)--(10,4.5);
    \draw[ultra thick] (5,3.5)--(10,3.5);
    \draw[ultra thick] (5,2.5)--(10,2.5);
    \draw[ultra thick,dashed] (12.5,6.5)--(17.5,6.5);
    \draw[ultra thick,dashed] (12.5,5.5)--(17.5,5.5);
    \draw[ultra thick,dashed] (12.5,4.5)--(17.5,4.5);
    \draw[ultra thick,dashed] (12.5,3.5)--(17.5,3.5);
    \draw[ultra thick,dashed] (12.5,2.5)--(17.5,2.5);

    \node[fill=white,inner sep=1pt] at (7.5,6.5) {\scriptsize $8$};
    \node[fill=white,inner sep=1pt] at (7.5,5.5) {\scriptsize $6$};
    \node[fill=white,inner sep=1pt] at (7.5,4.5) {\scriptsize $4$};
    \node[fill=white,inner sep=1pt] at (7.5,3.5) {\scriptsize $2$};
    \node[fill=white,inner sep=1pt] at (7.5,2.5) {\scriptsize $0$};
    \node[fill=white,inner sep=1pt] at (15,6.5) {\scriptsize $9$};
    \node[fill=white,inner sep=1pt] at (15,5.5) {\scriptsize $7$};
    \node[fill=white,inner sep=1pt] at (15,4.5) {\scriptsize $5$};
    \node[fill=white,inner sep=1pt] at (15,3.5) {\scriptsize $3$};
    \node[fill=white,inner sep=1pt] at (15,2.5) {\scriptsize $1$};
    
    \draw[ultra thick] (27.5,9)--(32.5,9);
    \draw[ultra thick] (27.5,8)--(32.5,8);
    \draw[ultra thick] (27.5,7)--(32.5,7);
    \draw[ultra thick] (27.5,6)--(32.5,6);
    \draw[ultra thick] (27.5,5)--(32.5,5);
    \draw[ultra thick] (27.5,4)--(32.5,4);
    \draw[ultra thick] (27.5,3)--(32.5,3);
    \draw[ultra thick] (27.5,2)--(32.5,2);
    \draw[ultra thick] (27.5,1)--(32.5,1);
    \draw[ultra thick] (27.5,0)--(32.5,0);

    \node[fill=white,inner sep=1pt] at (30,9) {\scriptsize $1$};
    \node[fill=white,inner sep=1pt] at (30,8) {\scriptsize $3$};
    \node[fill=white,inner sep=1pt] at (30,7) {\scriptsize $5$};
    \node[fill=white,inner sep=1pt] at (30,6) {\scriptsize $7$};
    \node[fill=white,inner sep=1pt] at (30,5) {\scriptsize $9$};
    \node[fill=white,inner sep=1pt] at (30,4) {\scriptsize $8$};
    \node[fill=white,inner sep=1pt] at (30,3) {\scriptsize $6$};
    \node[fill=white,inner sep=1pt] at (30,2) {\scriptsize $4$};
    \node[fill=white,inner sep=1pt] at (30,1) {\scriptsize $2$};
    \node[fill=white,inner sep=1pt] at (30,0) {\scriptsize $0$};

    \draw[->,  thick] (-2.5,4.5)--(2.5,4.5);
    \draw[->,  thick] (20,4.5)--(25,4.5);

    \node at (0,2.5) {\small into piles};
    \node at (22.5,2.5) {\small flip over};
\end{tikzpicture}}

\newtheorem{theorem}{Theorem}

\newtheorem{lemma}{Lemma}
\newtheorem{proposition}{Proposition}
\theoremstyle{definition}
\newtheorem*{problem}{Problem}

\begin{document}

\title{The mathematics of the flip and horseshoe shuffles}
\date{\empty}

\author{Steve Butler\thanks{Dept.\ of Mathematics, Iowa State University, Ames, IA 50011, USA. {\tt butler@iastate.edu}} \and Persi Diaconis\thanks{Dept.\ of Statistics and Dept.\ of Mathematics, Stanford University, Stanford, CA 94305, USA. {\tt diaconis@math.stanford.edu}} \and Ron Graham\thanks{Dept.\ of Computer Science and Engineering, UC San Diego, La Jolla, CA 92093, USA. {\tt graham@ucsd.edu}}}

\maketitle

\begin{abstract}
We consider new types of perfect shuffles wherein a deck is split in half, one half of the deck is ``reversed'', and then the cards are interlaced.  Flip shuffles are when the reversal comes from  flipping the half over so that we also need to account for face-up/face-down configurations while horseshoe shuffles are when the order of the cards are reversed but all cards still face the same direction.  We show that these shuffles are closely related to faro shuffling and determine the order of the associated shuffling groups.
\end{abstract}

\section{Introduction}
We were led to the problems in this note by a question from a magician friend, Jeremy Rayner, who sent a picture similar to Figure~\ref{fig:picture} along with a description of a new form of perfect shuffling \cite{rayner}.  

\begin{figure}[hftb]
\centering
\includegraphics[width=4.5in]{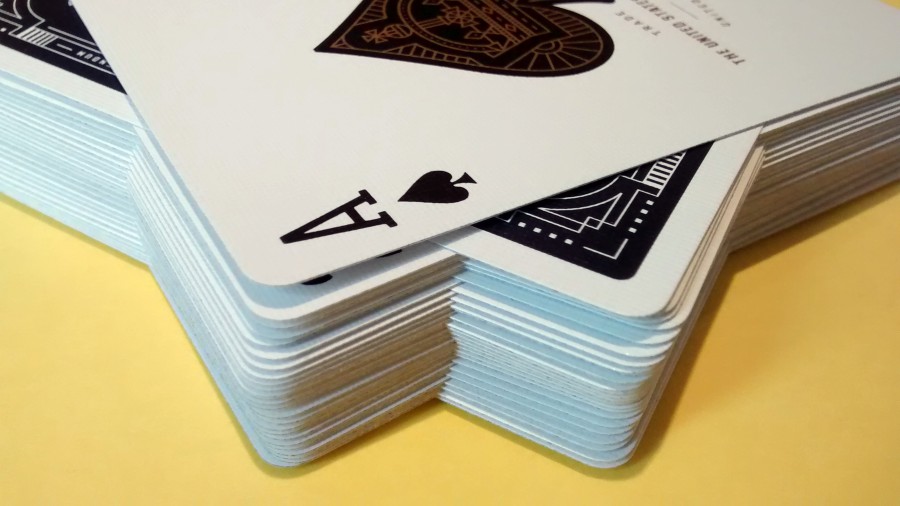}
\caption{A deck in the process of being flip shuffled}
\label{fig:picture}
\end{figure}

Crooked gamblers and skillful magicians have learned to shuffle cards perfectly; cutting off exactly half the deck and riffling the two halves together so they exactly alternate.  A more careful description and review of the traditional form of perfect shuffling is in Section~\ref{sec:faro}.  Our magician friend wanted to know what happens if one of the halves is turned face-up before shuffling.  We will call this a \emph{flip shuffle}.  Figure~\ref{fig:flip1} shows what happens for the (out) flip shuffle with a $10$ card deck (labeled $0,1,\ldots,9$ from top to bottom), and  dashed lines indicate a card has been turned over.

\begin{figure}[hft]
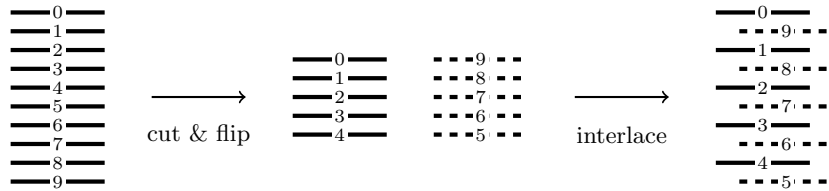

\centering
\picA
\caption{An out flip shuffle with $10$ cards}
\label{fig:flip1}
\end{figure}

Carefully repeating this $18$ times will return the deck of $10$ cards to the original order and same face-up/face-down configuration.  (One of the more studied questions related to perfect shuffles is the number of times needed to return a deck to the starting order.)

We note that normal perfect shuffles are not easy to perform (even for seasoned professionals), and this additional element of flipping cards over makes it even more difficult.  However, it is \emph{easy} to perform inverse shuffles.  To carry out the inverse of the (out) flip shuffle (i.e., as shown in Figure~\ref{fig:flip1}) we start with the deck and then deal it alternately between two piles, turning each card over before placing it on its pile.  Finally, turn the entire first pile over and place it on the second pile.  For the $10$ card deck this process is shown in Figure~\ref{fig:flip2} and the interested reader can again verify that it will take $18$ of these operations to return the deck to its starting configuration.

\begin{figure}[hft]
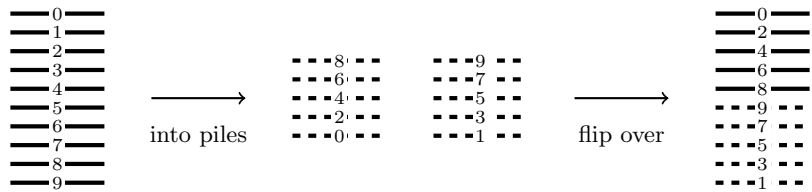

\centering
\picB
\caption{An inverse out flip shuffle with $10$ cards}
\label{fig:flip2}
\end{figure}

There is a second type of flip shuffle depending on how we choose to interlace.  What we have discussed so far corresponds to the \emph{out} flip shuffle (since the original top card remains on the \emph{out}side).  There is also an \emph{in} flip shuffle (where the original top card now winds up on the \emph{in}side).  The in flip shuffle for $10$ cards is shown in Figure~\ref{fig:flip3}.  An inverse flip shuffle deals into two piles as before but finishes by turning the second pile over and placing it on the first pile.  This is shown for $10$ cards in Figure~\ref{fig:flip4}.

\begin{figure}[hft]
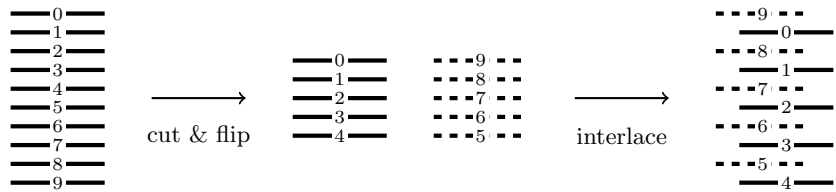

\centering
\picC
\caption{An in flip shuffle with $10$ cards}
\label{fig:flip3}
\end{figure}

\begin{figure}[hft]
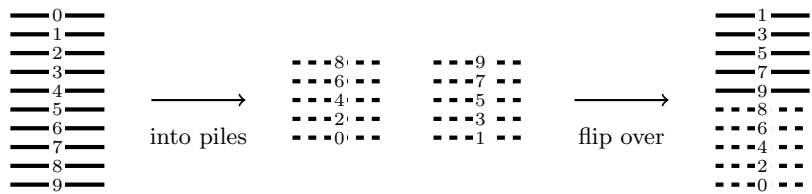

\centering
\picD
\caption{An inverse in flip shuffle with $10$ cards}
\label{fig:flip4}
\end{figure}

The `in' flips behave differently from the `out' flips, and the two can be combined in various ways to achieve different effects.  This leads to the main question in the field, namely, what can be done by combining these two types of shuffles.  The main result of this note is to completely determine the number of different orderings that can be achieved with these shuffling operations.

We first start by doing a review of what is known for ``ordinary'' perfect shuffles, also known as faro shuffling in Section~\ref{sec:faro}.  We then address what happens for flip shuffling in Section~\ref{sec:flip}.  Then in Sections~\ref{sec:horseshoe} and \ref{sec:poweroftwo} we will look at horseshoe shuffling, a face-down only version of flip shuffling.  In Section~\ref{sec:other} we make further connections with other types of shuffling, including a way to combine these different shuffling operations together before finally giving concluding remarks in Section~\ref{sec:conclusion}.

\section{Faro shuffling}\label{sec:faro}
The faro shuffle is the classic way of performing perfect shuffles.  This is carried out by splitting a deck of $2n$ cards into two equal piles of size $n$ and then perfectly interlacing the two piles.  Again there are two possibilities, either the top card remains on the outside (an out shuffle) or it moves to the second position (an in shuffle).  The two types are illustrated in Figure~\ref{fig:inout} for a deck with $10$ cards.

\begin{figure}[hft]
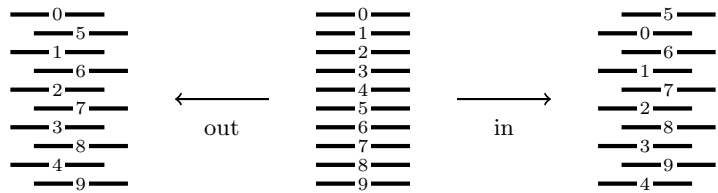

\centering
\picE
\caption{Examples of in and out faro shuffles for a deck with 10 cards}
\label{fig:inout}
\end{figure}

Faro shuffles are among the most well known shuffles both for mathematicians and magicians alike.  For example, starting with a standard deck of $52$ cards then performing eight consecutive out shuffles will return the deck to the starting order.  Combining these shuffles it is also possible to move the top card to any position in the deck, to do this label the positions as $0,1,2,\ldots$ and write the desired position in binary.  Now perform shuffles by reading from left to right where a ``$1$'' is an in shuffle and a ``$0$'' is an out shuffle.  So for example to move the card into the position $20$ we write this in binary as $10100$ and so perform the following shuffles: in, out, in, out, out.  Conversely there is a way to use in and out shuffles to move a card in any position in the deck to the top (see \cite{elmsley, elmsley2}).

A natural question is to determine what orderings of the deck are possible, and what sequence of shuffles leads to a desirable ordering.  Thinking of these shuffles as operators on the deck of cards then this problem becomes determining the ``shuffle group'', i.e., the group which acts on the deck generated by the in and out shuffles.  This work for faro shuffling was previously carried out (see \cite{DGK}), and the following was determined.

\begin{theorem}[Diaconis-Graham-Kantor] \label{thm:faro}
Let $Faro(2n)$ denote the shuffle group generated using in and out faro shuffles on a deck of $2n$ cards.  Then the following holds:
\[
|Faro(2n)|=\left\{
\begin{array}{r@{\quad}l}
2^9{\cdot}3{\cdot}5&\text{if $2n=12$,}\\
2^{17}{\cdot}3^3{\cdot}5{\cdot}11&\text{if $2n=24$,}\\
k2^{k\phantom{-1}}&\text{if $2n=2^k$,}\\
n!2^{n\phantom{-1}}&\text{if $n\equiv2\pmod4$ and $2n\ne 4,12$,}\\
n!2^{n-1}&\text{if $n\equiv1\pmod2$,}\\
n!2^{n-2}&\text{if $n\equiv0\pmod4$ and $2n\ne24,2^k$.}
\end{array}\right.
\]
(Note $|Faro(2n)|$ is the number of possible arrangements of a deck of $2n$ cards.)
\end{theorem}

We note that the groups themselves were determined by Diaconis, Graham and Kantor.  For the exceptional cases of $2n=12$ and $24$ the corresponding groups are $S_5\rtimes (\mathbb{Z}_2)^6$ and $M_{12}\rtimes (\mathbb{Z}_2)^{11}$, respectively, where $S_5$ is the symmetric group on $5$ elements and $M_{12}$ is the Mathieu group.  When $2n=2^k$ then the group is $(\mathbb{Z}_2)^k\rtimes \mathbb{Z}_k$, this case is particularly useful for performing magic tricks (more on this later).  Finally, all remaining cases correspond to subgroups of index $1$, $2$ or $4$ of the Weyl group $B_n$ (formed by the $2^nn!$ signed permutation matrices of order $n$).

An important property of faro shuffling we will make use of is ``stay stack''.  This means that two cards which start symmetrically opposite across the center will continue to stay symmetrically opposite after performing an in or out faro shuffles.  This can be easily checked for $2n=10$ by examining Figure~\ref{fig:inout}.  In general, if we label the positions as $0,1,2,\ldots$ then an out shuffle sends a card in position $i<2n-1$ to position $2i\pmod{2n-1}$ and it keeps the card in position $2n-1$ at $2n-1$.  In particular the top and bottom cards still stay on the top and bottom (i.e., are still symmetrically opposite) and otherwise the cards at positions $j$ and $2n-1-j$ map to positions $2j\pmod{2n-1}$ and $2(2n-1-j)\equiv -2j\pmod{2n-1}$ which are still symmetrically opposite (i.e., their positions add to $2n-1$).  Finally we note that an in shuffle can be carried out by temporarily adding a top and bottom card, performing an out shuffle, and then removing the top and bottom card.  Therefore since the stay stack result is true for out shuffles it is also true for in shuffles.

\section{Flip shuffling}\label{sec:flip}
The methods of carrying out flip and faro shuffling are similar.  So the easiest way to begin to evaluate flip shuffling is to establish a link with faro shuffling.  This is the goal of this section.

\begin{theorem}\label{thm:bijection} Let $Flip(2n)$ denote the shuffle group generated using in and out flip shuffles on a deck of $2n$ cards.  Then $Flip(2n)=Faro(4n)$.
\end{theorem}

\begin{proof}
First we give a bijection between a deck with $4n$ cards with the cards in stay stack and a deck with $2n$ cards where we keep track of whether a card is face-up or face-down. We will use ``$\overline{*}$'' to denote the card ``${*}$'' has been turned over.  We label our $4n$ cards now as $1,\overline{1},2,\overline{2},\ldots,2n,\overline{2n}$ and assume that the cards $i$ and $\overline{i}$ are symmetrically opposite across the center.  The bijection is now to take the deck of $4n$ cards and simply keep the first $2n$ cards, conversely given $2n$ cards with a face-up/face-down pattern we form our deck of $4n$ cards by first placing these in the first $2n$ positions and then in the second half placing cards to satisfy stay stack.

Examples of this bijection for $2n=10$ are as follows:
\[
\begin{array}{c@{\,\,}c@{\,\,}c@{\,\,}c@{\,\,}c@{\,\,}c@{\,\,}c@{\,\,}c@{\,\,}c@{\,\,}c@{\,\,}c@{\,\,}c@{\,\,}c@{\,\,}c@{\,\,}c@{\,\,}c@{\,\,}c@{\,\,}c@{\,\,}c@{\,\,}c@{~~\longleftrightarrow~~}c@{\,\,}c@{\,\,}c@{\,\,}c@{\,\,}c@{\,\,}c@{\,\,}c@{\,\,}c@{\,\,}c@{\,\,}c}
0&1&2&3&4&5&6&7&8&9&\overline{9}&\overline{8}&\overline{7}&\overline{6}&\overline{5}&\overline{4}&\overline{3}&\overline{2}&\overline{1}&\overline{0}&0&1&2&3&4&5&6&7&8&9\\
\overline{9}&0&\overline{8}&1&\overline{7}&2&\overline{6}&3&\overline{5}&4&\overline{4}&5&\overline{3}&6&\overline{2}&7&\overline{1}&8&\overline{0}&9&\overline{9}&0&\overline{8}&1&\overline{7}&2&\overline{6}&3&\overline{5}&4\\
\overline{4}&\overline{9}&5&0&\overline{3}&\overline{8}&6&1&\overline{2}&\overline{7}&7&2&\overline{1}&\overline{6}&8&3&\overline{0}&\overline{5}&9&4&\overline{4}&\overline{9}&5&0&\overline{3}&\overline{8}&6&1&\overline{2}&\overline{7}\\
7&\overline{4}&2&\overline{9}&\overline{1}&5&\overline{6}&0&8&\overline{3}&3&\overline{8}&\overline{0}&6&\overline{5}&1&9&\overline{2}&4&\overline{7}&7&\overline{4}&2&\overline{9}&\overline{1}&5&\overline{6}&0&8&\overline{3}\\
3&7&\overline{8}&\overline{4}&\overline{0}&2&6&\overline{9}&\overline{5}&\overline{1}&1&5&9&\overline{6}&\overline{2}&0&4&8&\overline{7}&\overline{3}&3&7&\overline{8}&\overline{4}&\overline{0}&2&6&\overline{9}&\overline{5}&\overline{1}
\end{array}
\]

This bijection commutes with the shuffling operations, i.e., performing a faro shuffle on a deck of $4n$ cards and then applying the bijection will produce the same arrangement as applying the bijection and performing the corresponding flip shuffle.  This can be seen in the sequence of above examples where we have performed in  shuffles on both sides.

In general we observe that when taking the deck with $4n$ cards and cutting it in half, the bottom half exactly corresponds to the action of taking the top half and turning it over (i.e., the order is reversed as well as the face-up/face-down status of each card).  When we now interlace these cards the first $2n$ cards are found by interlacing the first $n$ cards on the top half (which corresponds to the top half of the deck we want to flip shuffle) and the first $n$ cards on the bottom half (which corresponds to the bottom half of the deck we want to flip shuffle) and therefore the first $2n$ cards are exactly the same as we would expect.  Now using the property of stay stack the second $2n$ cards are as the bijection requires.

Since the bijection commutes with the shuffling operations, then for any deck arrangement that we can achieve by using faro shuffles on the $4n$ cards there is a unique corresponding arrangement on the deck using flip shuffles with $2n$ cards.  In particular the group actions are isomorphic and so the shuffling groups are the same.
\end{proof}

\section{Horseshoe shuffling}\label{sec:horseshoe}
Now we turn to the variation of flip shuffling where we no longer keep track of the face-up/face-down pattern of the cards.  This can be carried out as a flip shuffle on a special deck called a ``mirror deck'' which prints the numbers of the cards on both sides.  On a normal deck this is carried out by cutting the deck in half and reversing the order of the bottom half before interlacing.  As before we have two types of horseshoe shuffles, in and out, and they are both shown in Figure~\ref{fig:smale} for a deck with $10$ cards. 

\begin{figure}[hft]
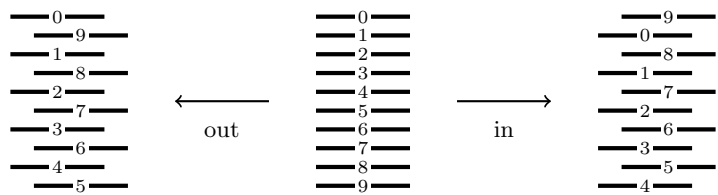

\centering
\picF
\caption{The effect of in and out horseshoe shuffles for a deck with 10 cards}
\label{fig:smale}
\end{figure}

The name ``horseshoe'' comes from being a discrete analog of the horseshoe map introduced by Stephen Smale (see \cite{smale}), which stretches out the unit square and then folds it back in onto itself.

As with flip shuffles, inverse horseshoe shuffles are easy to do by dealing two piles on the table.  For an inverse out horseshoe shuffle, deal two piles, turning the top card over and placing it onto the table in a first pile, dealing the second card straight down into a second pile and continuing, over (onto the first pile), down (onto the second pile), over, down, \dots. Finish by picking up the first pile, and turning it over onto the second pile.  This results in the original top card being back on top (see Figure~\ref{fig:invhorseout}). For an inverse in horseshoe shuffle deal into two piles but now deal down onto the first pile, over onto the second pile, down, over, down, over, \dots.  Finish by picking up the second pile, and turning it over onto the first pile.  This results in the original top card going to the bottom (see Figure~\ref{fig:invhorsein}).   These are easy for most performers (and spectators, see our card trick below) to perform.

\begin{figure}[hft]
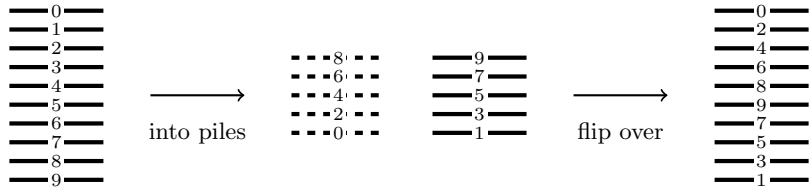

\centering
\picAAA
\caption{An inverse out horseshoe shuffle with $10$ cards}
\label{fig:invhorseout}
\end{figure}

\begin{figure}[hft]
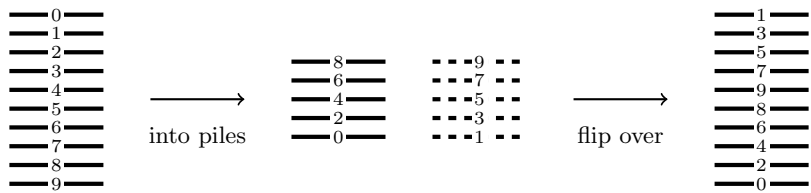

\centering
\picBBB
\caption{An inverse in horseshoe shuffle with $10$ cards}
\label{fig:invhorsein}
\end{figure}

Since horseshoe shuffling and flip shuffling differ by whether or not we keep track of orientation we have that the corresponding groups for the horseshoe shuffle are subgroups of those for the flip shuffle.  As we will show, the groups are among the simplest and most beautiful in the theory of perfect shuffling.

\begin{theorem}\label{thm:horseshoe}
Let $Horse(2n)$ denote the shuffle group generated using in and out horseshoe shuffles on a deck of $2n$ cards.  Then the following holds:
\[
|Horse(2n)|=\left\{
\begin{array}{r@{\quad}l}
4{\cdot}5{\cdot}6&\text{if $2n=6$,}\\
8{\cdot}9{\cdot}10{\cdot}11{\cdot}12&\text{if $2n=12$,}\\
(k+1)2^{k}&\text{if $2n=2^k$ with $k\ge 2$,}\\
(2n)!&\text{if $n\equiv1\pmod2$ and $n\ne 3$,}\\
(2n)!/2&\text{if $n\equiv0\pmod2$ and $n\ne6,2^k$.}
\end{array}\right.
\]
(Note $|Horse(2n)|$ is the number of possible arrangments of a deck of $2n$ cards.)
\end{theorem}

The special cases of $2n=6$ and $2n=12$ were verified by computer, and we will not discuss them further here.  We will however remark that these two cases are somewhat extraordinary and that the corresponding groups are known.  When $2n=6$ then the first \emph{three} cards can be arbitrary which then forces the order of the  last three cards, the corresponding group is $PGL(2,5)$ which is isomorphic to $S_5$.  When $2n=12$ then the first \emph{five} can be arbitrary which then forces the order of the last seven cards, the corresponding group is again the Mathieu group $M_{12}$ (which is sharply $5$-transitive).

We will establish Theorem~\ref{thm:horseshoe} by working through the remaining general cases.  The main thing to observe is that if we can achieve an ordering of the cards using flip shuffles then we also can achieve the same ordering of the cards using horseshoe shuffles (i.e., the difference between the two being whether or not we keep track of the orientation of the cards).

\begin{lemma}
If $n\equiv1\pmod{2}$ and $n\ne 3$ then $Horse(2n)=S_{2n}$, the symmetric group of order $2n$.  In particular, $|Horse(2n)|=(2n)!$.
\end{lemma}
\begin{proof}
By Theorem~\ref{thm:bijection} we have that $|Flip(2n)|=|Faro(4n)|=(2n)!2^{2n}$.  On the other hand there are only $(2n)!$ ways to order the deck and $2^{2n}$ face-up/face-down configurations so that the total number of possible orderings under flip shuffling is $(2n)!2^{2n}$ showing that each combination must occur.  In particular, we must be able to get any ordering using horseshoe shuffles.
\end{proof}

\begin{lemma}
If $n\equiv 0\pmod{2}$ and $n\ne 6,2^k$ then $Horse(2n)=A_{2n}$, the alternating group of order $2n$.  In particular,  $|Horse(2n)|=(2n)!/2$.
\end{lemma}
\begin{proof}
By Theorem~\ref{thm:bijection} we have that $|Flip(2n)|=|Faro(4n)|=(2n)!2^{2n-2}$.  We now take a closer look at what orderings are possible and what face-up/face-down configurations are possible.

For the orderings we will think of our shuffling operations as applying a permutation to $0,1,\ldots,2n-1$.  We claim that the two different shuffling operations both  correspond to even permutations (i.e., can be written as the product of an even number of transpositions).  If this is the case then it will follow that all possible arrangements that we can form correspond to even permutations and in particular there are at most $(2n)!/2$ possible orderings.  

To verify the claim we start with the cards labeled $0,1,\ldots,n,n+1,\ldots,2n-1$ and we want to take it to the ordering $0,2n-1,1,\ldots,n+1,n-1,n$.  The cards $n-1$ and $n$ are in the correct relative ordering.  We now move the card $n+1$ so it comes before $n-1$ which takes two transpositions (i.e., exchange $n$ and $n+1$ and then exchange $n-1$ and $n+1$).  We now move the card $n+2$ so it comes before $n-2$ which takes four transpositions (i.e., we exchange places with $n,n-1,n+1,n-2$).  This continues in general so that we will move card $n+i$ so it comes before $n-i$ which will take $2i$ transpositions (i.e., we exchange places with $n,n-1,\ldots,n+(i-1),n-i$).  Therefore the total number of transpositions we have used is 
\[
0+2+4+\cdots+2(n-1)=n(n+1),
\]
which is even and this is an even permutation.  For the in shuffle we first place it into the out shuffle and then switch every pair, i.e., apply the transpositions $(0,2n-1)$, $(1,2n-2)$, \ldots, $(n-1,n)$, and since by hypothesis $n$ is even there are an even number of such pairs showing that the in shuffle will also be an even permutation.

As for the face-up/face-down configuration let $X_i$ be an indicator that the card in the $i$th position is face-up (i.e., $X_i$ is $1$ if the card is face-up and $0$ if the card is face-down).  Therefore the number of cards which are face-up is $\sum X_i$.  Applying either shuffle we have that the number of cards which are face-up after one operation is
\begin{multline*}
X_0+X_1+\cdots+X_{n-1}+(1-X_n)+(1-X_{n+1})+\cdots+(1-X_{2n-1})
\\ \equiv X_0+X_1+\cdots+X_{2n-1}+n\pmod2.
\end{multline*}
By hypothesis we have that $n$ is even and therefore the parity of the number of face-up cards never changes with shuffling.  Since we start with an even number of face-up cards we will always have an even number of face-up cards and therefore there are at most $2^{2n-1}$ possible face-up/face-down configurations.

Combining the two above ideas we see that there are at most $(2n)!2^{2n-2}$ possible configurations under flip shuffling, but we already know that we have that many achievable configurations.  And therefore each possible combination of ordering (coming from even permutations) and face-up/face-down cards (as long as an even number of face-up cards) is possible.  In particular, we can achieve any ordering for horseshoe shuffling which corresponds to an even permutation.
\end{proof}

All that remains now to establish Theorem~\ref{thm:horseshoe} is the case when the deck size is a power of $2$ which we will handle in the next section.

\section{Horseshoe shuffles for deck sizes of powers of $2$}\label{sec:poweroftwo}
Combining Theorem~\ref{thm:bijection} and Theorem~\ref{thm:faro} we can conclude that when $2n=2^k$ that $|Flip(2n)|=(k+1)2^{k+1}$.  Of course this takes more information into account and so we will work to understand what is possible when we ignore the face-up/face-down orientation of the cards.  In this section we will determine this, and moreover we will show that knowing just the first two cards we can determine the order of the remaining cards.

The first observation is that by using binary to represent the position of a card we can determine where the card will go when we horseshoe shuffle.  There are two types of shuffles and on each of the two shuffles we have a top half and a bottom half.  If we represent the location using binary as $x_{k-1}x_{k-2}\ldots x_0$ (so the top half is $x_{k-1}=0$ and the bottom half is $x_{k-1}=1$) then the following rules can be verified (note by flipping the bits in the second half we reverse their order):
\begin{align*}
\text{\emph{in} shuffle:}&\quad x_{k-1}x_{k-2}\ldots x_1x_0\longrightarrow\left\{
\begin{array}{r@{\quad}l}
x_{k-2}\ldots x_0x_{k-1}&\text{if }x_{k-1}=0\\[5pt]
\overline{x_{k-2}\ldots x_0x_{k-1}}&\text{if }x_{k-1}=1
\end{array}\right.\\[10pt]
\text{\emph{out} shuffle:}&\quad x_{k-1}x_{k-2}\ldots x_1x_0\longrightarrow\left\{
\begin{array}{r@{\quad}l}
x_{k-2}\ldots x_0x_{k-1}&\text{if }x_{k-1}=0\\[5pt]
\overline{x_{k-2}\ldots x_0}x_{k-1}&\text{if }x_{k-1}=1
\end{array}\right.
\end{align*}
Where we let $\overline{x_i}=1-x_i$, i.e., flipping the bit.

\begin{proposition}
For any $0\le i,j\le 2^k-1$ we can use $k$ horseshoe shuffles on a deck of size $2^k$ to move the card in the $i$th position to the $j$th position.  In particular, every card can be moved to the top in at most $k$ shuffles.
\end{proposition}

This can be seen by noting that through the first $k$ shuffles the location of the card in the top or bottom half is completely determined by where the card initially started.  In particular we will consistently either apply the rule for the top or bottom half as we move through the $k$ shuffles and in each there are $2$ outcomes, making $2^k$ outcomes after $k$ shuffles.  Since we can always reverse the process at each step (i.e., we know which half of the deck we were in at the previous step and the last digit indicated whether we were an in or out shuffle) the outcomes are distinct.  But there are only $2^k$ possible locations and so in $k$ shuffles we can move a card to any position.

An illustration of this is shown in Figure~\ref{fig:ksteps} which demonstrates what goes on if we look at the position for the card initially in position $1011$ through four shuffles where an \emph{in} shuffle is up and an \emph{out} shuffle is down as we go from left to right.

\begin{figure}[hft]
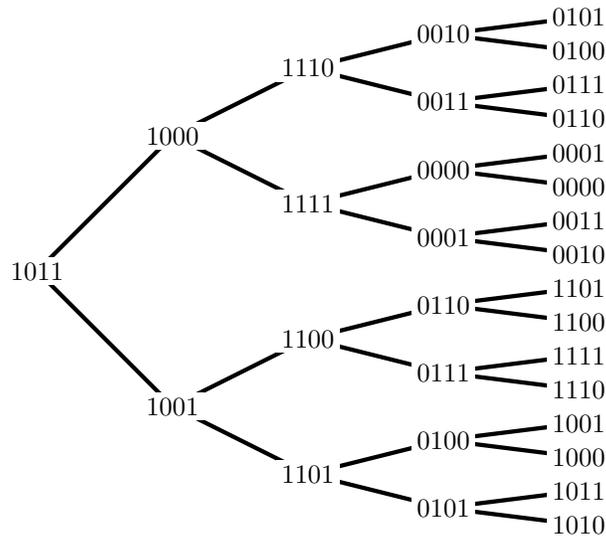

\centering
\picBB
\caption{The position of the card initially at $1011$ through four shuffles}
\label{fig:ksteps}
\end{figure}

With this we can now get a lower bound for the number of arrangements.

\begin{proposition}
There are at least $(k+1)2^k$ different arrangements of a deck of size $2^k$ using in and out horseshoe shuffles.
\end{proposition}
\begin{proof}
By the previous proposition we can place any of the $2^k$ cards on the top position.  Once that card is on top then repeated out shuffles will leave that card on top.  Examining what happens to the card in the second position we note that it will go into positions as follows:
\[
1\longrightarrow 2\longrightarrow 2^2\longrightarrow\cdots\longrightarrow 2^{k-1}\longrightarrow (2^k-1)\longrightarrow 1
\]
In particular the cards in these $k+1$ positions will cycle through the second position and so for every card in the first position there are at least $k+1$ cards that can be in the second position, establishing the result.
\end{proof}

To show that we can only get this many different arrangements we will first show that the horseshoe shuffles preserve a special ordering.  This is based off of the following diagram.

\bigskip
\noindent\hfil
\picAA
\bigskip

The way to use this is to start with some card (with face value labeled using binary) in the first position, and then choose some starting point in the diagram.  From the starting card we will now determine the ordering of the remaining cards where at each stage we duplicate the existing cards but then change either the indicated bit in our new duplicate set or complement as indicated in the diagram and then move to the next point of the diagram on the diagram and repeat.

As an example of this process, if the deck has $16$ cards and the first card is eleven (or in binary is $1011$) and the starting point is to change the bit corresponding to $2^2$ then this process will work as shown in Table~\ref{tab:process}.  This produces the ordering
\[
11,~15,~3,~7,~4,~0,~12,~8,~10,~14,~2,~6,~5,~1,~13,~9.
\]

\begin{table}[htf]
\caption{An example of determining the ordering}
\label{tab:process}
\small
\[
\begin{array}{|c|c|c|c|c|}\hline
\text{~~~~~initial~~~~~}&\text{change $2^2$ bit}&\text{change $2^3$ bit}&\text{~complement~}&\text{change $2^0$ bit}\\ \hline
1011&1011&1011&1011&1011\\
    &1111&1111&1111&1111\\
    &    &0011&0011&0011\\
    &    &0111&0111&0111\\
    &    &    &0100&0100\\
    &    &    &0000&0000\\
    &    &    &1100&1100\\
    &    &    &1000&1000\\
    &    &    &    &1010\\
    &    &    &    &1110\\
    &    &    &    &0010\\
    &    &    &    &0110\\
    &    &    &    &0101\\
    &    &    &    &0001\\
    &    &    &    &1101\\
    &    &    &    &1001\\ \hline
\end{array}
\]
\end{table}

To understand what is happening we note that upon splitting the deck in half and reversing the bottom half then the following two things are true:
\begin{itemize}
\item Each half deck is built off of the first $k-1$ operations used to form the original deck.  (This follows by the construction.)
\item The two half decks differ from each other exactly corresponding to a change in the diagram that was not used (i.e., if we started by flipping in the $2^2$ bit then the two differ in the $2^1$ bit).  This follows by noting that the card in the last position will have gone through all but one of the operations which when combined has the net effect of the missing operation, now combining this with the construction we see that the second to last entry differs in the first instruction, the prior two in the second instruction and so on.
\end{itemize}
With these two properties we see that when we shuffle the two halves together (either in or out) then the first pair will differ in the missing operation, then the next pair is formed by copying and applying the prior first operation, and so on. 

In particular if we started with a special ordering then we will maintain a special ordering after we shuffle.  Since we start with the cards in the order $0,1,\ldots,2^n-1$, which is special, then the only possible orderings are these special orderings.  But there are at most $2^k(k+1)$ such orderings, i.e., pick the first card and then pick one of $k+1$ operations to get the second.

Finally we mention that if we reversed the order of the deck that we would still have the same order of operations used to build up the sequence, i.e., $9$, $13$, $1$, $5$, \dots can be built by starting with $9$, then changing the $2^2$ bit, the $2^3$ bit, \dots.  We will use this in the next section in an application.

\section{Other shuffling methods}\label{sec:other}
The horseshoe shuffle is closely related to the \emph{milk shuffle} which consists of taking the deck and then ``milking'' off the top and bottom card and putting them on the table, then continue milking off the top and bottom card and placing on top of the pile on the table until all the cards have been used.  So for example if we have a deck with $2n=10$ cards the result is what is shown in Figure~\ref{fig:milk}.

\begin{figure}[hft]
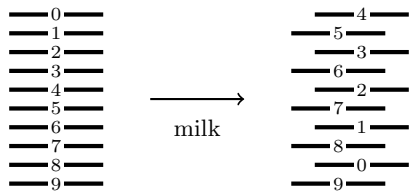

\centering
\picG
\caption{The effect of the milk shuffle for a deck with 10 cards}
\label{fig:milk}
\end{figure}

This looks very similar to Figure~\ref{fig:smale}.  In particular, if you draw a line connecting the cards in their original order you will see that the horseshoe forms the shape ``${\cup}$'' while the milk forms the shape ``$\cap$''.  So these shuffles are essentially the same, and one can be done by first flipping the deck, carrying out the other form of shuffling, and then flipping the deck back over.  This suggests there should be two milk shuffles, the second one consists of milking off the top and bottom cards and then before putting them on the pile switching the order, an inconvenient shuffle for performance.

Going further this same connection of flipping  also applies for the inverse of the milk shuffle, the \emph{Monge shuffle}.  This is carried out by taking the deck and placing the top card in the other hand, then we feed the top card of the deck into the other hand one card at a time alternating between being placed on the top or the bottom of the pile.  (When the second card is placed under the first this corresponds to the inverse of the milk shuffle, though the second card can also be placed over the first.)  This is illustrated in Figure~\ref{fig:Monge} for a deck with $2n=10$ cards.

\begin{figure}[hft]
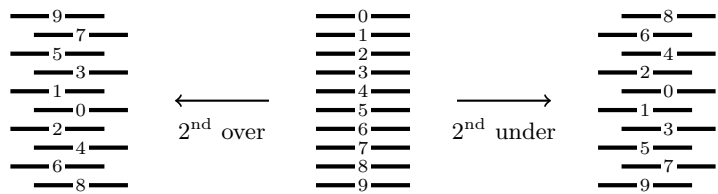

\centering
\picH
\caption{The effect of the two Monge shuffles for a deck with $10$ cards}
\label{fig:Monge}
\end{figure}

The milk and Monge shuffles have been extensively studied before (see \cite{magic,DGK,2power} and references therein; there are also many online videos which demonstrate how to perform these shuffles).

Returning to the previous section we note that the special ordering we discussed is preserved when we apply an inverse shuffling operation as well as when we flip the deck.  Therefore for deck sizes of $2^k$ then any application of combinations of the horseshoe shuffle, milk shuffle, Monge shuffle, or reversing the deck (i.e., dealing the deck into a pile) will still preserve our special ordering.  This can be used for a very simple trick which we outline here.

\subsection*{Working from the outside in}
We first describe the effect.  We start with a packet of eight cards and tell our audience ``In this small pile we have eight cards which we will now use to practice various shuffling techniques.''  At this point the deck is handed over to audience members where they are then taught about the different shuffles, i.e., the horseshoe, the milk, the Monge, and dealing the cards onto a pile (i.e., to reverse the order).  They are allowed and encouraged to do as many shuffles as they want in whatever order they choose.  

After much shuffling we now declare ``By now we have become somewhat expert as to how to shuffle the deck, but of course we have been shuffling so long that we have no idea of the current ordering of the cards.  But as with many problems we can solve what the ordering should be by working from the outside in.  So now please deal the cards down either from left to right or from right to left.''  At this point the cards are dealt (of course the dealing as well as all of the shuffling can occur with the performer having their back turned).  ``To figure out what cards we have we need to start somewhere and so will you please turn over the end cards.''  The end cards are turned over, and the magic begins!

Let us suppose that what we see is the following:
\[
4~?~?~?~?~?~?~6
\]
We continue, ``Aha! the $4$ and the $6$, an interesting combination.  The $4$ is a power of $2$ so wants to be with another power of $2$ so the next card could be a $2$ or an $8$ and in this case it is an $8$, similarly the number $6$ is composite so it wants to be next to one of its prime divisors which are $2$ or $3$ and in this case it is next to the $2$ \dots'' this continues until all the cards have been declared, with various dubious or numerological reasons given as to what should come next, and cards are turned over as the declarations are made with each one being correct.  Until finally the ordering is given
\[
4 ~8~ 3 ~7~ 5 ~A~ 2 ~6.
\]
``And so we see that we can solve this problem by working from the outside in!''

The way this is done is to take the cards and initially order them $8A234567$ (either from top to bottom or bottom to top).  We will treat $8$ as $0$ for our computations and so what we have actually done is place the cards in increasing order, and in particular an ordering from Section~\ref{sec:poweroftwo} that obeys the rule of the diagram.  Then as the various shuffles are performed the ordering will change, but each time it will still obey the same diagram and so that once the cards are placed down it only remains to determine a starting card (and as noted we can start from either end) and where in the diagram we start from.  Taking the end cards we compute in binary how they differ, since $4$ and $6$ in binary are respectively $100$ and $110$ we see they differ in the $2^1$ bit (in general if the numbers add to $7$ they are complements otherwise the difference tells you which bit to work in).  Therefore we should start with the \emph{next} operation in our diagram which in this case is to flip the $2^2$ bit and so $100$ (i.e., the $4$) becomes $000$ (i.e., which would normally be $0$ but we represented by $8$), similarly on the other side $110$ becomes $010$ (i.e., the $2$) giving us our next layer in.  Now to get the remaining four we use the four cards showing and the next operation (in this case complementation) to determine what will come next.

This can be done fairly quickly with a bit of practice in binary, remembering that $8$ corresponds with $0$.  The biggest danger is that someone makes a mistake in shuffling in which case the ordering will be off.  So make sure to carefully walk the audience through each kind of shuffle so that they will not make a mistake; alternatively, you can carry out the shuffles as directed by audience members.

\section{Concluding remarks}\label{sec:conclusion}
As we have seen the flip and horseshoe shuffles have nice connections with known shuffles. As a result similar performance techniques can be used for these shuffles.  For instance to move the top card into the $i$th position (where the positions have been labeled as $0,1,\ldots$), then we again write $i$ in binary and translate that into in and out shuffles appropriately.  This works because as we move the card down to the $i$th position it will always remain in the top half as we shuffle and so behaves precisely like what happens with faro shuffling.  On the other hand the problem of moving a card from the $i$th position to the top position, known as Elmsley's problem (see \cite{elmsley, elmsley2}), will need different techniques.  For the deck with $10$ cards the shortest ways to move the card in the $i$th position to the top are given in Table~\ref{tab:elmsley}

\begin{table}[hft]
\caption{How to move the $i$th card to the top}
\label{tab:elmsley}
\small
\[
\begin{array}{|c|l|}\hline
\text{Position}&\text{Shuffling sequence}\\\hline\hline
1&\text{out, in, out, in ~~~~\textbf{or}~~~~ out, out, in, in}\\ \hline
2&\text{in, out, in ~~~~\textbf{or}~~~~ out, in, in} \\ \hline
3&\text{in, out, out, in ~~~~\textbf{or}~~~~ in, in, in, in}\\ \hline
4&\text{in, in} \\ \hline
5&\text{out, in} \\ \hline
6&\text{out, out, out, in ~~~~\textbf{or}~~~~ out, in, in, in} \\ \hline
7&\text{out, out, in ~~~~\textbf{or}~~~~ in, in, in} \\ \hline
8&\text{in, in, out, in ~~~~\textbf{or}~~~~ in, out, in, in} \\ \hline
9&\text{in} \\ \hline
\end{array}
\]
\end{table}

\begin{problem}
Determine a general method for a deck of $2n$ cards that moves the card in the $i$th position to the top using in and out horseshoe shuffles.
\end{problem}

We also note that we have carried out the analysis for decks with an even number of cards.  There are also two types of perfect flip and horseshoe shuffles for decks of odd order.  We leave the analysis of this case as an exercise for the interested reader.

One annoying problem is that we couldn't identify the group generated by in and out horseshoe shuffles for decks of size $2^k$. We know it has order $(k+1)2^k$ for $k$ larger than $1$, it is $S_2$ when $k =1$ and the alternating group $A_4$ when $k = 2$, what is it for larger k?

The shuffles we have mentioned here are some of many possible shuffles which exhibit interesting mathematical properites.  More information about the mathematics behind various shuffles and how it can be used to amaze and entertain friends and colleagues can be found in the recent books of Diaconis and Graham \cite{magic} as well as Mulcahy \cite{colm}.


\bigskip
\noindent\textbf{Acknowledgment:}  The authors thank Jeremy Rayner for introducing them to flip shuffles.

\end{document}